\documentclass[11pt]{article}

\usepackage[a4paper,margin=1.08in]{geometry}
\usepackage{amsmath,amssymb,amsthm,mathtools,bbm}
\usepackage{microtype}
\usepackage{cite}
\usepackage[hidelinks]{hyperref}
\hypersetup{
  pdftitle={A Fourier-hypercontractive proof of the sharp bound for cancellative pairs},
  pdfauthor={Fan Chang},
  pdfkeywords={cancellative pairs, one-sided noise operator, hypercontractivity, biased Fourier analysis, Sinkhorn scaling}
}

\allowdisplaybreaks

\newtheorem{theorem}{Theorem}[section]
\newtheorem{lemma}[theorem]{Lemma}

\theoremstyle{remark}
\newtheorem{remark}[theorem]{Remark}

\newcommand{\E}{\mathbb{E}}
\newcommand{\T}{\mathrm{T}}

\newcommand{\1}{\mathbf{1}}

\newcommand{\A}{\mathcal{A}}
\newcommand{\B}{\mathcal{B}}

\newcommand\tup[1]{\left\langle #1 \right\rangle}

\title{A hypercontractive proof of the sharp bound for cancellative pairs}
\author{Fan Chang\thanks{\raggedright School of Statistics and Data Science, Nankai University, Tianjin, China. Email:~\texttt{1120230060@mail.nankai.edu.cn}. Supported by the National Natural Science Foundation of China under grant 124B2019.}}
\date{}

\begin{document}

\maketitle

\begin{abstract}
Fang and Huang~\cite{FangHuang2026} proved that every cancellative pair $(\mathcal{A},\mathcal{B})$ of families of subsets of $[n]$ satisfies $|\mathcal{A}|\cdot|\mathcal{B}|\leq(\frac{9}{4})^n$. Their proof uses entropy. We give a Fourier-analytic proof based on Lifshitz's one-sided noise operator $\mathrm{T}^{1/4\to1/2}$~\cite{Lifshitz2020}. The main tool is a near-$L^1$ hypercontractive estimate for this operator.
\end{abstract}

\section{Introduction}

A pair $(\mathcal A,\mathcal B)$ of families of subsets of $[n]$ is \emph{cancellative} if $A\cup B=A'\cup B$ implies $A=A'$ for all $A,A'\in\mathcal A$ and $B\in\mathcal B$, and $A\cup B=A\cup B'$ implies $B=B'$ for all $A\in\mathcal A$ and $B,B'\in\mathcal B$.

Cancellative pairs were introduced by Holzman and K\"orner~\cite{HolzmanKorner1995}. They proved $|\mathcal A|\cdot|\mathcal B|\leq2.3264^n$, and Janzer~\cite{Janzer2018} improved the constant to $2.2682$. Tolhuizen~\cite{Tolhuizen2000} constructed symmetric cancellative pairs with product size $(\frac{9}{4}-o(1))^n$. Very recently, Fang and Huang~\cite{FangHuang2026} proved the matching upper bound.

\begin{theorem}[Fang--Huang]\label{thm:main}
If $(\A,\B)$ is a cancellative pair of families of subsets of $[n]$, then
$$
|\A|\cdot|\B|\leq\left(\frac94\right)^n.
$$
\end{theorem}

Fang and Huang apply Sinkhorn's theorem to the matrix $3^{-|A\cup B|}$ and interpret the scaled matrix as a coupling with uniform marginals on $\A$ and $\B$. They then enlarge the random union by a binary $Z$-channel. The two random sets are conditionally independent given the enlarged set, and a one-coordinate entropy inequality completes the proof.

One motivation for this work is to understand the elegant entropy proof of Fang and Huang: which steps depend on cancellativity, and which entropy-based tools can be replaced. Fourier-analytic and information-theoretic proofs on product spaces often suggest two formulations of the same inequality. Examples include Gross's equivalence between hypercontractivity and logarithmic Sobolev inequalities~\cite{Gross1975}, and Friedgut's information-theoretic proof of a hypercontractive inequality~\cite{Friedgut2024}. For a fixed joint distribution, hypercontractivity has equivalent formulations in terms of relative entropy and mutual information, with close connections to strong data-processing and $\Phi$-Sobolev inequalities~\cite{AhlswedeGacs1976,AnantharamEtAl2013,Nair2014,NairWang2016,Raginsky2016,BeigiGohari2018}. More generally, Carlen and Cordero-Erausquin~\cite{CarlenCorderoErausquin2009} proved a duality between Brascamp--Lieb inequalities and entropy subadditivity, including best constants and equality cases, and Liu, Courtade, Cuff, and Verd\'u~\cite{LiuEtAl2018} extended this duality to forward--reverse inequalities with stochastic maps. Yu~\cite{Yu2021} has developed this approach further, deriving nonlinear Brascamp--Lieb and hypercontractive inequalities by information-theoretic and coupling methods; related work with Anantharam and Chen~\cite{YuAnantharamChen2024} gives stronger hypercontractive inequalities. The monograph of Yu and Tan~\cite[Part~III]{YuTan2022} summarizes these connections from both the information-theoretic and Fourier-analytic viewpoints. In the present problem, we retain the Sinkhorn scaling from Fang and Huang's proof and replace the subsequent entropy argument by a near-$L^1$ hypercontractive estimate and standard H\"older/Cauchy--Schwarz inequalities.

The operator we use is Lifshitz's one-sided noise operator $\T^{1/4\to1/2}$~\cite{Lifshitz2020}. Related operators between differently biased cubes appeared in the work of Ahlberg, Broman, Griffiths, and Morris on noise sensitivity in continuum percolation~\cite{AhlbergEtAl2014}. Lifshitz used them in robust sharp threshold arguments. We follow the notation $\T^{p\to q}$ of Kalai and Lifshitz~\cite{KalaiLifshitz2025}, who use the thinning operator to compare Shapley values and threshold intervals.

Given $T\subseteq[n]$, retain each element of $T$ independently with probability $1/2$, and denote the resulting subset by $W$. Then
$$
\T^{1/4\to1/2}f(T)=\E[f(W)\mid T]
=\frac{1}{2^{|T|}}\sum_{W\subseteq T}f(W).
$$
Let $\mu_\alpha$ denote the $\alpha$-biased product measure on $\{0,1\}^n$. For $r\geq1$, write $\|f\|_{r,\alpha}=(\E_{\mu_\alpha}|f|^r)^{1/r}$. The following near-$L^1$ estimate is the main tool of the proof.

\begin{theorem}[Near-$L^1$ hypercontractivity]\label{thm:hyper}
For every $1/2<\gamma<1$, there exists $\varepsilon_0=\varepsilon_0(\gamma)>0$ such that the following holds. If $0<\varepsilon<\varepsilon_0$, $p=1+\varepsilon$, and $q=1+\gamma\varepsilon$, then
\begin{equation}
\|\T^{1/4\to1/2} f\|_{p,1/2}\leq\|f\|_{q,1/4}
\end{equation}
for every $n\geq1$ and every $f:\{0,1\}^n\to\mathbb R$.
\end{theorem}

\begin{remark}\label{rem:quarter-bias}
The bias $1/4$ comes from both the probabilistic definition of $W$ and the weight assigned to the union.
Indeed, let $T\sim\mu_{1/2}$ and, conditional on $T$, obtain
$W\subseteq T$ by retaining each element of $T$ independently with
probability $1/2$. For each coordinate $i$,
$$
\mathbb{P}(i\in W)=\mathbb{P}(i\in T)\mathbb P(i\in W\mid i\in T)=\frac14,
$$
and the coordinates of $W$ remain independent. Hence $W\sim\mu_{1/4}$,
and consequently
$$
\E_{\mu_{1/2}}[\T^{1/4\to1/2} f] =\E_{\mu_{1/4}}[f].
$$
Thus $\T^{1/4\to1/2}$ maps functions from the
$1/4$-biased cube to the uniform cube.

The same bias arises from the weight $3^{-|W|}$ that will be used
later for the union $W=A\cup B$. Indeed,
$$
\mu_{1/4}(W)=\left(\frac14\right)^{|W|}
\left(\frac34\right)^{n-|W|}=\left(\frac34\right)^n\frac{1}{3^{|W|}}.
$$
More generally, a product measure $\mu_\alpha$ has density proportional
to $3^{-|W|}$ precisely when
$\frac{\alpha}{1-\alpha}=\frac{1}{3}$, which gives $\alpha=\frac{1}{4}$.

There is also a direct combinatorial reason for using this operator. Fix $A\in\mathcal A$. By cancellativity, the map $B\mapsto A\cup B$ is injective on $\mathcal B$. If $h_A$ is supported
on these unions, then
$$
(\T^{1/4\to1/2} h_A)(T)=\frac{1}{2^{|T|}}\sum_{\substack{B\in\mathcal B\\ A\cup B\subseteq T}}
h_A(A\cup B)=\frac{1}{2^{|T|}}\mathbbm{1}_{A\subseteq T}\sum_{\substack{B\in\mathcal B\\ B\subseteq T}}
h_A(A\cup B),
$$
where the second equality follows from $A\cup B\subseteq T$ if and only if $A\subseteq T$ and $B\subseteq T$. Thus averaging over the subsets of $T$ converts a condition on the union $A\cup B$ into separate containment conditions on $A$ and $B$. This is the form needed for the row and column estimates in the proof.
\end{remark}
\paragraph{Proof overview.}
The proof idea may be summarized as follows: we use cancellativity to construct suitable test functions, and then apply Theorem~\ref{thm:hyper} to these functions. One side of the resulting inequalities is evaluated by Sinkhorn's theorem (Theorem~\ref{thm:sinkhorn}), while the other is converted into the desired product bound by Cauchy--Schwarz and H\"older inequalities.

Write $a=|\A|$ and $b=|\B|$. The proof begins with the observation of Holzman and K\"orner~\cite{HolzmanKorner1995}: cancellativity is needed only to ensure that $B\mapsto A\cup B$ is injective for each fixed $A$, and that $A\mapsto A\cup B$ is injective for each fixed $B$. These maps allow us to define the row and column test functions on the union variable. The same two injectivity statements are also used in~\cite{Janzer2018,FangHuang2026}.

Sinkhorn's theorem balances the kernel $3^{-|A\cup B|}$: there are positive weights $r_A,s_B$ such that
$$
\sum_{B\in\mathcal B}r_As_B3^{-|A\cup B|}=\frac1a,
\qquad
\sum_{A\in\mathcal A}r_As_B3^{-|A\cup B|}=\frac1b.
$$
For each $A\in\A$, define a test function supported on the sets
$A\cup B$, with value $(ar_As_B)^{1/q}$ at $A\cup B$. The row normalization
determines its $L^q(\mu_{1/4})$ norm. Applying Theorem~\ref{thm:hyper} and summing over $A$ gives a row estimate;
interchanging $\A$ and $\B$ gives the corresponding column estimate. Cauchy--Schwarz combines the two, while
Lemma~\ref{lem:powers} and a weighted H\"older inequality reduce the
remaining expression to the Sinkhorn normalization. Letting $\varepsilon\downarrow0$ and then $\gamma\downarrow1/2$ makes the one-coordinate factor tend to $3/2$, yielding the bound $(9/4)^n$.

\medskip
\noindent\emph{Organization.}
This paper is organized as follows. Section~\ref{sec:preliminaries} reviews biased Fourier analysis and the one-sided noise operator, and records the power-sum inequality and Sinkhorn's theorem used in the proof. Section~\ref{sec:proof} proves Theorems~\ref{thm:hyper} and~\ref{thm:main}. Section~\ref{sec:remarks} compares our argument with the entropy proof of Fang and Huang.

\section{Preliminaries}\label{sec:preliminaries}

\subsection{Fourier analysis on the \texorpdfstring{$p$}{p}-biased cube}

We identify subsets $A\subseteq[n]$ with their indicator vectors $\1_A\in\{0,1\}^n$, where $(\1_A)_i=1$ if and only if $i\in A$. Under this identification, we write $f(A)$ for $f(\1_A)$. For $0<p<1$, the $p$-biased measure $\mu_p$ on $\{0,1\}^n$ is
\[
\mu_p(x)=p^{|x|}(1-p)^{n-|x|}.
\]
For $r\geq1$, write $\|f\|_{r,p}=(\mathbb E_{\mu_p}|f|^r)^{1/r}$ and $\tup{f,g}_p=\mathbb E_{\mu_p}[fg]$.

For $i\in[n]$, define
\[
\chi_i^p(x)=\frac{x_i-p}{\sqrt{p(1-p)}},
\]
and for $S\subseteq[n]$, let $\chi_S^p(x)=\prod_{i\in S}\chi_i^p(x)$, with $\chi_\varnothing^p=1$. The functions $\{\chi_S^p:S\subseteq[n]\}$ form an orthonormal basis of $L^2(\{0,1\}^n,\mu_p)$. Thus
$$
f=\sum_{S\subseteq[n]}\hat{f}_p(S)\chi_S^p,\qquad
\hat{f}_p(S)=\tup{f,\chi_S^p}_p,
$$
and Parseval's identity reads $\|f\|_{2,p}^2=\sum_S\hat{f}_p(S)^2$. We refer to~\cite{ODonnell2014} for background.

\subsection{The one-sided noise operator}

Let $0<\alpha<\beta<1$. Given $x\in\{0,1\}^n$, form $y\leq x$ by retaining each coordinate with $x_i=1$ independently with probability $\frac{\alpha}{\beta}$. The one-sided noise operator is
$$
\T_n^{\alpha\to\beta}f(x)=\mathbb E[f(y)\mid x]=\sum_{y\leq x}\left(\frac\alpha\beta\right)^{|y|}
\left(\frac{\beta-\alpha}{\beta}\right)^{|x|-|y|}f(y).
$$
If $x\sim\mu_\beta$, then $y\sim\mu_\alpha$, and hence
$$
\E_{\mu_\beta}[\T^{\alpha\to\beta}f]
=\E_{\mu_\alpha}[f].
$$
The construction is coordinatewise, so the operator tensorizes. For $g:\{0,1\}\to\mathbb R$,
\begin{equation}\label{eq:one-bit-operator}
(\T_1^{\alpha\to\beta}g)(0)=g(0),
\qquad
(\T_1^{\alpha\to\beta}g)(1)=\frac{\beta-\alpha}{\beta}g(0)+\frac{\alpha}{\beta}g(1).
\end{equation}
For $i\in[n]$, let $\T^{(i)}$ apply this rule in the $i$th coordinate and leave the other coordinates fixed. The coordinate operators commute, and
\begin{equation}\label{eq:operator-tensorization}
\T_n^{\alpha\to\beta}=\T^{(1)}\cdots\T^{(n)}=\bigl(\T_1^{\alpha\to\beta}\bigr)^{\otimes n}.
\end{equation}
For example, if $f(x)=\prod_{i=1}^n f_i(x_i)$, then
$$
(\T_n^{\alpha\to\beta}f)(x)
=\prod_{i=1}^n(\T_1^{\alpha\to\beta}f_i)(x_i).
$$
We will use this coordinatewise factorization in the Fourier calculation below and in the proof of Theorem~\ref{thm:hyper}.

The usual noise operator is self-adjoint on a fixed biased cube. In contrast, $\T^{\alpha\to\beta}$ averages over subsets and maps functions from $(\{0,1\}^n,\mu_\alpha)$ to functions on $(\{0,1\}^n,\mu_\beta)$. It acts diagonally between the two biased Fourier bases~\cite{Lifshitz2020,KalaiLifshitz2025}.

\begin{lemma}\label{lem:fourier}
Let $0<\alpha<\beta<1$ and set $\rho=\sqrt{\frac{\alpha(1-\beta)}{\beta(1-\alpha)}}$. If $f=\sum_{S\subseteq[n]}\hat{f}_\alpha(S)\chi_S^\alpha$, then
$$
\T_n^{\alpha\to\beta}f
=\sum_{S\subseteq[n]}\rho^{|S|}\hat{f}_\alpha(S)\chi_S^\beta.
$$
\end{lemma}

\begin{proof}
In one coordinate, $\T_1^{\alpha\to\beta}1=1$ and a direct calculation from~\eqref{eq:one-bit-operator} gives
$$
\T_1^{\alpha\to\beta}\chi_1^\alpha
=\rho\chi_1^\beta.
$$
Then tensorization gives
$$
\T_n^{\alpha\to\beta}\chi_S^\alpha
=\rho^{|S|}\chi_S^\beta.
$$
The claimed expansion follows by linearity.
\end{proof}

We henceforth suppress the subscript $n$ when the dimension is clear. For $\T=\T^{1/4\to1/2}$ and $A\subseteq[n]$, we have
\begin{equation}
  \T f(A)=\frac{1}{2^{|A|}}\sum_{W\subseteq A}f(W)
\end{equation}
and
\begin{equation}
\T\chi_S^{1/4}=\left(\frac{1}{\sqrt{3}}\right)^{|S|}\chi_S^{1/2}.
\end{equation}
The adjoint of $\T$ is described by the reverse procedure. Start from $A\sim\mu_{1/4}$, keep every element of $A$, and add each element outside $A$ independently with probability $1/3$. The resulting set has law $\mu_{1/2}$. This is the enlargement used in~\cite{FangHuang2026}; we return to its relation with $\T$ in Section~\ref{sec:remarks}.

\subsection{A power-sum inequality}
The next elementary inequality will be used when the row and column estimates are combined.

\begin{lemma}\label{lem:powers}
Let $1<q<p$ be such that $\frac{p(q-1)}{p-q}\geq1$. Then for every finite sequence of nonnegative numbers $z_1,\ldots,z_N$,
\begin{equation}
\left(\sum_{i=1}^Nz_i^{\frac{1}{q}}\right)^{\frac{p}{2}}\left(\sum_{i=1}^Nz_i^{\frac{p}{q}}\right)^{\frac{1}{2}}\geq\left(\sum_{i=1}^Nz_i\right)^{\frac{p}{q}}.
\end{equation}
\end{lemma}

\begin{proof}
If all $z_i$ vanish, the inequality is trivial. By homogeneity, we may
assume that $\sum_{i=1}^Nz_i=1$. Let $\varphi(t)=\log\left(\sum_{i=1}^N z_i^t\right)$. By H\"older's inequality, for $s,t>0$ and $0\le \alpha\le1$,
$$
\sum_{i=1}^N z_i^{(1-\alpha)s+\alpha t}
\le\left(\sum_{i=1}^N z_i^s\right)^{1-\alpha}\left(\sum_{i=1}^N z_i^t\right)^\alpha.
$$
Thus $\varphi$ is convex. Since $\frac{1}{q}<1<\frac{p}{q}$ and $\varphi(1)=0$, convexity gives
$$
\frac{\varphi(1)-\varphi(\frac{1}{q})}{1-\frac{1}{q}}\le \frac{\varphi(\frac{p}{q})-\varphi(1)}{\frac{p}{q}-1}\Rightarrow  \frac{\varphi(\frac{1}{q})}{1-\frac{1}{q}}\ge \frac{-\varphi(\frac{p}{q})}{\frac{p}{q}-1}.
$$
Moreover, $\varphi(\frac{p}{q})\le0$, because $\frac{p}{q}>1$ and $\sum_i z_i=1$. Hence
$$
p\varphi\left(\frac{1}{q}\right)+\varphi\left(\frac{p}{q}\right)\ge \left(1-\frac{p(q-1)}{p-q}\right)\varphi\left(\frac{p}{q}\right)\ge0.
$$
Exponentiating one half of this inequality yields
$$
\left(\sum_{i=1}^Nz_i^{\frac{1}{q}}\right)^{\frac{p}{2}}\left(\sum_{i=1}^Nz_i^{\frac{p}{q}}\right)^{\frac{1}{2}}\ge1,
$$
which is the desired inequality under the normalization $\sum_i z_i=1$.
\end{proof}

\subsection{Sinkhorn's theorem}

Sinkhorn scaling says that a positive matrix can be balanced by multiplying each row and each column by a positive scalar. We need the rectangular form in which the balanced matrix has uniform row and column marginals~\cite{Sinkhorn1964,Sinkhorn1967}.
\begin{theorem}[Sinkhorn]\label{thm:sinkhorn}
Let $K=(K_{ij})$ be a positive $m\times k$ matrix. There exist positive numbers $r_1,\ldots,r_m$ and $s_1,\ldots,s_k$ such that the scaled matrix $P_{ij}=r_iK_{ij}s_j$ satisfies
\begin{itemize}
    \item $\sum_{j=1}^k P_{ij}=\frac1m$ for every $i\in[m]$;
    \item $\sum_{i=1}^m P_{ij}=\frac1k$ for every $j\in[k]$.
\end{itemize}
\end{theorem}
In particular, $\sum_{i,j}P_{ij}=1$, so $P$ may be viewed as a probability distribution whose two marginals are uniform.
\begin{proof}
Consider
$$
\Psi(y_1,\ldots,y_k)
=\frac1m\sum_{i=1}^m\log\left(\sum_{j=1}^kK_{ij}e^{y_j}\right)-\frac1k\sum_{j=1}^ky_j.
$$
The function is invariant under adding the same constant to every $y_j$, so restrict it to the hyperplane $\sum_jy_j=0$. On this hyperplane,
\[
\Psi(y)\geq\log\left(\min_{i,j}K_{ij}\right)+\max_jy_j,
\]
and hence $\Psi$ is coercive. Let $y$ be a minimizer. The first-order condition gives
\[
\frac1m\sum_{i=1}^m
\frac{K_{ij}e^{y_j}}{\sum_{\ell=1}^kK_{i\ell}e^{y_\ell}}
=\frac1k
\quad\text{for every }j.
\]
Set $s_j=e^{y_j}$ and
\[
r_i=\frac1{m\sum_{\ell=1}^kK_{i\ell}s_\ell}.
\]
The row identities follow from the definition of $r_i$, and the first-order condition gives the column identities.
\end{proof}

For the union kernel $K_{A,B}=3^{-|A\cup B|}$, the scaled entries
$$
P_{A,B}=r_As_B3^{-|A\cup B|}
$$
form a probability distribution with uniform marginals on $\A$ and $\B$. The row identities normalize the row test functions, and the column identities normalize the column test functions. Thus the same weights $r_A,s_B$ can be used in both estimates before they are combined.

\section{Proofs}\label{sec:proof}

\subsection{Proof of Theorem~\ref{thm:hyper}}
The standard proof of the Bonami--Beckner hypercontractive inequality on the discrete cube has two steps: one proves a two-point inequality in dimension one and then tensorizes it to the full cube by Minkowski's inequality; see~\cite{Bonami1970,Beckner1975} and~\cite[Sections~9.3--9.4]{ODonnell2014}. Our argument follows the same tensorization argument. 

\begin{proof}[Proof of Theorem~\ref{thm:hyper}]
Write $\T=\T^{1/4\to1/2}$. We first prove the estimate on $\{0,1\}$. Since $\T$ is positive, $|\T f|\leq \T|f|$, so it suffices to take $f\geq0$. The case $f=0$ is immediate; otherwise, by homogeneity, assume $\E_{\mu_{1/4}}[f]=1$. Then
$$
f_t(0)=1-t,\qquad f_t(1)=1+3t
$$
for some $-\frac{1}{3}\leq t\leq1$. Moreover,
$$
\T f_t(0)=1-t, \qquad \T f_t(1)=1+t.
$$

Let $p=1+\varepsilon$ and $q=1+\gamma\varepsilon$, and define
$$
\Delta_\varepsilon(t)=\log\|\T f_t\|_{p,1/2}-\log\|f_t\|_{q,1/4}.
$$
At $\varepsilon=0$, both norms are equal to $1$. Differentiation at $\varepsilon=0$ gives
$$
\lim_{\varepsilon\downarrow0}\frac{\Delta_\varepsilon(t)}{\varepsilon}
=D_\gamma(t),
$$
where, with $0\log0=0$,
$$
D_\gamma(t)=\frac12(1-t)\log(1-t)+\frac12(1+t)\log(1+t)-\gamma\left[\frac34(1-t)\log(1-t)
+\frac14(1+3t)\log(1+3t)\right].
$$
We claim that $D_{1/2}(t)\leq0$ on $[-\frac{1}{3},1]$. Indeed, let $F(t)=8D_{1/2}(t)$. Then
$$
F(t)=(1-t)\log(1-t)+4(1+t)\log(1+t)
-(1+3t)\log(1+3t),
$$
and
$$
F'(t)=\log\frac{(1+t)^4}{(1-t)(1+3t)^3},
\qquad
F''(t)=\frac{-4(3t-1)}{(t-1)(t+1)(3t+1)}.
$$
Note that the function $F$ is concave on $[-\frac{1}{3},\frac{1}{3}]$. Since $F(0)=F'(0)=0$, its graph lies below the tangent line at $0$, and hence $F\leq0$ on this interval. On $[\frac{1}{3},1]$, the function is convex, while $F(\frac{1}{3})<0$ and $F(1)=0$. It therefore lies below the chord joining these two endpoint values, so $F\leq0$ there as well.

For $\gamma>1/2$,
$$
D_\gamma(t)=D_{1/2}(t)-\left(\gamma-\frac12\right)\left[\frac34(1-t)\log(1-t)+\frac14(1+3t)\log(1+3t)\right].
$$
Write $J(t)=\frac34(1-t)\log(1-t)
+\frac14(1+3t)\log(1+3t)$. Since $\frac34(1-t)+\frac14(1+3t)=1$, the strict convexity of $x\mapsto x\log x$ gives
$$
J(t)\ge\left[\frac34(1-t)+\frac14(1+3t)\right]\log\left[\frac34(1-t)+\frac14(1+3t)\right]=0.
$$
Equality holds only when $1-t=1+3t$, that is, $t=0$. Consequently, for $t\ne0$,
$$
D_\gamma(t)<D_{1/2}(t)\le0.
$$

It remains to pass from the derivative at $\varepsilon=0$ to a norm inequality for small positive $\varepsilon$. Define
$$
\Phi(\varepsilon,t)=
\begin{cases}
\Delta_\varepsilon(t)/\varepsilon,&\varepsilon>0,\\
D_\gamma(t),&\varepsilon=0.
\end{cases}
$$
For some $\varepsilon_1>0$, the function $\Phi$ is continuous on $[0,\varepsilon_1]\times[-\frac{1}{3},1]$. This follows from the continuous extensions of $x^r$ and $x^r\log x$ at $x=0$, for $r$ in a neighbourhood of $1$. The first two derivatives of $\Phi$ with respect to $t$ are continuous near $(0,0)$. For every $\varepsilon>0$,
$$
\Phi(\varepsilon,0)=\partial_t\Phi(\varepsilon,0)=0,
\qquad
\partial_t^2\Phi(\varepsilon,0)=1-3\gamma<0.
$$
After decreasing $\varepsilon_1$, the function $t\mapsto\Phi(\varepsilon,t)$ is concave on a fixed neighbourhood of $0$ and is nonpositive there. On the complement of that neighbourhood, $D_\gamma$ has a strictly negative maximum. Uniform continuity then gives $\Phi(\varepsilon,t)<0$ for all sufficiently small $\varepsilon$. This proves the one-dimensional estimate.

The $n$-dimensional estimate now follows by induction. For $n>1$, write $x=(x',x_n)$, let $\T^{(n)}$ act on the last coordinate, and set $g=\T^{(n)}f$. Since $\T_n=\T_{n-1}\T^{(n)}$, the induction hypothesis in the first $n-1$ coordinates, Minkowski's inequality, and the one-bit estimate in the last coordinate give
\begin{equation*}
    \begin{split}
\|\T_n f\|_{p,1/2}&=\left(\underset{x_n\sim\mu_{1/2}}{\E}\left[\|\T_{n-1}g(\,\cdot,x_n)\|_{p,1/2}^{p}\right]\right)^{1/p}\leq\left(\underset{x_n\sim\mu_{1/2}}{\E}\left[\|g(\,\cdot,x_n)\|_{q,1/4}^{p}\right]\right)^{1/p}\\
&\leq\left(\underset{x'\sim\mu_{1/4}^{\otimes(n-1)}}{\E}\left[\|g(x',\cdot)\|_{p,1/2}^{q}\right]\right)^{1/q}\leq\left(\underset{x'\sim\mu_{1/4}^{\otimes(n-1)}}{\E}\left[  \|f(x',\cdot)\|_{q,1/4}^{q}\right]\right)^{1/q}=\|f\|_{q,1/4}.
\end{split}
\end{equation*}
The second inequality is Minkowski in the form $L^p(L^q)\leq L^q(L^p)$ for $p\geq q$. This completes the induction.
\end{proof}

\subsection{Proof of Theorem~\ref{thm:main}}
We now apply Theorem~\ref{thm:hyper} to the test functions defined by the Sinkhorn-scaled union kernel.
\begin{proof}[Proof of Theorem~\ref{thm:main}]
Let $a=|\A|$ and $b=|\B|$. The conclusion is immediate if either family is empty, so assume $a,b>0$.
Apply Theorem~\ref{thm:sinkhorn} to the matrix $K_{A,B}=3^{-|A\cup B|}$ where $A\in \A$ and $B\in \B$. There are positive weights $r_A$ and $s_B$ such that
$$
\sum_{B\in\mathcal B}\frac{r_A\cdot s_B}{3^{|A\cup B|}}=\frac1a, \qquad
\sum_{A\in\mathcal A}\frac{r_A\cdot s_B}{3^{|A\cup B|}}=\frac1b.
$$
In particular,
\begin{equation}\label{eq:total}
\sum_{A\in\A}\sum_{B\in \B}\frac{r_A\cdot s_B}{3^{|A\cup B|}}=1.
\end{equation}

Fix $\frac{1}{2}<\gamma<1$ and choose $\varepsilon>0$ sufficiently small for Theorem~\ref{thm:hyper}. Fix $A\in\A$. Since $(\A,\B)$ is cancellative, the map $B\mapsto A\cup B$ is injective. Define $h_A:\{0,1\}^n\to[0,\infty)$ by
$$
h_A(U)=\begin{cases}
(ar_As_B)^{\frac{1}{q}},&U=A\cup B\text{ for some }B\in\mathcal B,\\
0,&\text{otherwise}.
\end{cases}
$$
Using the row normalization and $\mu_{1/4}(U)=(\frac{3}{4})^n\frac{1}{3^{|U|}}$, we obtain
$$
\|h_A\|_{q,1/4}^q=a\left(\frac34\right)^n\sum_{B\in\mathcal B}\frac{r_A\cdot s_B}{3^{|A\cup B|}}=\left(\frac34\right)^n.
$$
Moreover,
$$
\T h_A(C)=\frac{1}{2^{|C|}}\sum_{U\subseteq C}h_A(U)=\frac{1}{2^{|C|}}(ar_A)^{\frac{1}{q}}\mathbbm{1}_{\{A\subseteq C\}}
\sum_{\substack{B\in\mathcal B\\B\subseteq C}}s_B^{\frac{1}{q}}.
$$
Applying Theorem~\ref{thm:hyper} to $h_A$ with $p=1+\varepsilon$ and $q=1+\gamma\varepsilon$, and raising both sides to the $p$th power, gives
$$
\underset{C\sim\mu_{1/2}}{\E}
\left[\frac{1}{2^{p|C|}}(ar_A)^{\frac{p}{q}}\mathbbm{1}_{\{A\subseteq C\}}
\left(\sum_{\substack{B\in\mathcal B\\B\subseteq C}}s_B^{\frac{1}{q}}\right)^p
\right]
\leq\left(\frac34\right)^{n\frac{p}{q}}.
$$
Summing over $A\in\mathcal A$ and dividing by $a^{\frac{p}{q}}$ yields
\begin{equation}\label{eq:rows}
\underset{C\sim\mu_{1/2}}{\E}\left[
\frac{1}{2^{p|C|}}\left(\sum_{\substack{B\in\mathcal B\\B\subseteq C}}s_B^{\frac{1}{q}}\right)^p
\sum_{\substack{A\in\mathcal A\\A\subseteq C}}r_A^{\frac{p}{q}}
\right]
\leq a^{1-\frac{p}{q}}\left(\frac34\right)^{n\frac{p}{q}}.
\end{equation}

Interchanging $\mathcal A$ and $\mathcal B$, and using the injectivity of $A\mapsto A\cup B$ together with the column normalization, gives
\begin{equation}\label{eq:columns}
\underset{C\sim\mu_{1/2}}{\E}
\left[\frac{1}{2^{p|C|}}\left(\sum_{\substack{A\in\mathcal A\\A\subseteq C}}r_A^{\frac{1}{q}}\right)^p\sum_{\substack{B\in\mathcal B\\B\subseteq C}}s_B^{\frac{p}{q}}\right]
\leq b^{1-\frac{p}{q}}\left(\frac34\right)^{n\frac{p}{q}}.
\end{equation}

For each $C\subseteq[n]$, define $F(C)=2^{-p|C|}
\left(\sum_{\substack{B\in\mathcal B\\B\subseteq C}}s_B^{1/q}\right)^p
\sum_{\substack{A\in\mathcal A\\A\subseteq C}}r_A^{\frac{p}{q}}$ and define $G(C)$ by interchanging $\A$ and $\B$ and their weights. Multiplying~\eqref{eq:rows} and~\eqref{eq:columns} and applying Cauchy--Schwarz to $\sqrt{F}$ and $\sqrt{G}$ gives
\begin{equation}\label{eq:product-FG}
\left(\mathbb E_{\mu_{1/2}}\left[\sqrt{FG}\right]\right)^2\le\E_{\mu_{1/2}}[F]\E_{\mu_{1/2}}[G]
\leq(ab)^{1-\frac{p}{q}}\left(\frac34\right)^{2n\frac{p}{q}}.
\end{equation}
For a fixed $C$, note that
\begin{equation}\label{eq:CS}
    \begin{split}
\sqrt{F(C)G(C)}&=\frac{1}{2^{p|C|}}\left(\sum_{\substack{A\in\mathcal A\\A\subseteq C}}r_A^{\frac{1}{q}}\right)^{\frac{p}{2}}
\left(\sum_{\substack{A\in\mathcal A\\A\subseteq C}}r_A^{\frac{p}{q}}\right)^{\frac{1}{2}}\times
\left(\sum_{\substack{B\in\mathcal B\\B\subseteq C}}s_B^{\frac{1}{q}}\right)^{\frac{p}{2}}
\left(\sum_{\substack{B\in\mathcal B\\B\subseteq C}}s_B^{\frac{p}{q}}\right)^{\frac{1}{2}}\\
&\geq \frac{1}{2^{p|C|}} \left(\sum_{\substack{A\in\mathcal A\\A\subseteq C}}r_A\right)^{\frac{p}{q}}\left(\sum_{\substack{B\in\mathcal B\\B\subseteq C}}s_B\right)^{\frac{p}{q}},
    \end{split}
\end{equation}
where we apply Lemma~\ref{lem:powers} to the sequences $\{r_A:A\in\A,\ A\subseteq C\}$ and $\{s_B:B\in\B,\ B\subseteq C\}$ (The hypothesis of Lemma 2.2 holds since $\frac{p(q-1)}{p-q}=\frac{\gamma(1+\varepsilon)}{1-\gamma}>1$). Let
$$
Z(C)=
\left(\sum_{\substack{A\in\mathcal A\\A\subseteq C}}r_A\right)
\left(\sum_{\substack{B\in\mathcal B\\B\subseteq C}}s_B\right).
$$
Combining \eqref{eq:product-FG} and \eqref{eq:CS}, we obtain
\begin{equation}\label{eq:cs-bound}
(ab)^{1-\frac{p}{q}}\left(\frac34\right)^{2n\frac{p}{q}}
\geq\left(\E_{\mu_{1/2}}\left[\sqrt{FG}\right]\right)^2\geq\left(
\underset{C\sim\mu_{1/2}}{\E}
\left[\frac{1}{2^{p|C|}}Z(C)^\frac{p}{q}\right]
\right)^2.
\end{equation}

It remains to bound the last expectation. A $\mu_{1/3}$-random set contains a fixed set $U=A\cup B$ with probability $3^{-|U|}$. Hence \eqref{eq:total} implies
$$
\sum_{C\subseteq[n]}\mu_{1/3}(C)Z(C)=\sum_{A\in\mathcal A}\sum_{B\in\mathcal B}r_As_B
\sum_{C\supseteq A\cup B}\mu_{1/3}(C)=\sum_{A\in\A}\sum_{B\in \B}\frac{r_A\cdot s_B}{3^{|A\cup B|}}=1.
$$
For each $C$, split the summand as
$$
\mu_{1/3}(C)Z(C)
=\left[2^{\frac{-n-p|C|}{p/q}}Z(C)\right]\cdot\left[\mu_{1/3}(C)2^{\frac{n+p|C|}{p/q}}\right].
$$
H\"older's inequality with conjugate exponents $\frac{p}{q}$ and $\frac{\frac{p}{q}}{\frac{p}{q}-1}$ gives
\begin{equation*}
    \begin{split}
1&=\sum_{C\subseteq[n]}\mu_{1/3}(C)Z(C)=\sum_{C\subseteq[n]}\left[2^{\frac{-n-p|C|}{p/q}}Z(C)\right]\cdot\left[\mu_{1/3}(C)2^{\frac{n+p|C|}{p/q}}\right]\\
&\le\left(\sum_{C\subseteq[n]}2^{-n-p|C|}Z(C)^\frac{p}{q}\right)^{\frac{q}{p}}\cdot \left(\sum_{C\subseteq[n]}
\mu_{1/3}(C)^{\frac{\frac{p}{q}}{\frac{p}{q}-1}}2^{\frac{n+p|C|}{\frac{p}{q}-1}}\right)^{\frac{p/q-1}{p/q}}\\
&=\left(\underset{C\sim\mu_{1/2}}{\E}\left[\frac{Z(C)^\frac{p}{q}}{2^{p|C|}}\right]\right)^{\frac{q}{p}}\cdot \left(\sum_{C\subseteq[n]}\left(\frac{1}{3}\right)^{\frac{\frac{p}{q}|C|}{\frac{p}{q}-1}}\left(\frac{2}{3}\right)^{\frac{\frac{p}{q}(n-|C|)}{\frac{p}{q}-1}}2^{\frac{(p+1)|C|}{\frac{p}{q}-1}}2^{\frac{n-|C|}{\frac{p}{q}-1}}\right)^{\frac{p/q-1}{p/q}}\\
&=\left(\underset{C\sim\mu_{1/2}}{\E}\left[\frac{Z(C)^\frac{p}{q}}{2^{p|C|}}\right]\right)^{\frac{q}{p}}\cdot \left(\sum_{C\subseteq[n]}\left[\left(\frac{1}{3}\right)^{\frac{\frac{p}{q}}{\frac{p}{q}-1}}2^{\frac{p+1}{\frac{p}{q}-1}}\right]^{|C|}\left[\left(\frac{2}{3}\right)^{\frac{\frac{p}{q}}{\frac{p}{q}-1}}2^{\frac{1}{\frac{p}{q}-1}}\right]^{n-|C|}\right)^{\frac{p/q-1}{p/q}}\\
&=\left(\underset{C\sim\mu_{1/2}}{\E}\left[\frac{Z(C)^\frac{p}{q}}{2^{p|C|}}\right]\right)^{\frac{q}{p}}\cdot \left(\left(\frac{1}{3}\right)^{\frac{\frac{p}{q}}{\frac{p}{q}-1}}2^{\frac{p+1}{\frac{p}{q}-1}}+\left(\frac{2}{3}\right)^{\frac{\frac{p}{q}}{\frac{p}{q}-1}}2^{\frac{1}{\frac{p}{q}-1}}\right)^{n\frac{p/q-1}{p/q}}\\
&=\left(\underset{C\sim\mu_{1/2}}{\E}\left[\frac{Z(C)^\frac{p}{q}}{2^{p|C|}}\right]\right)^{\frac{q}{p}}\cdot \left[\left(\frac{1}{3}\right)^{\frac{\frac{p}{q}}{\frac{p}{q}-1}}2^{\frac{\frac{p}{q}+1}{\frac{p}{q}-1}}\left(1+2^{\frac{p(q-1)}{p-q}}\right)\right]^{n\frac{p/q-1}{p/q}}.
    \end{split}
\end{equation*}
Thus
\begin{equation}\label{eq:holder-bound}
\underset{C\sim\mu_{1/2}}{\E}\left[\frac{Z(C)^\frac{p}{q}}{2^{p|C|}}\right]
\geq
\left[3^{\frac{p}{q}}2^{-\frac{p}{q}-1}
\left(1+2^{\frac{p(q-1)}{p-q}}\right)^{-(\frac{p}{q}-1)}\right]^n.
\end{equation}
Substituting \eqref{eq:holder-bound} into \eqref{eq:cs-bound} gives
$$
(ab)^{1-\frac{p}{q}}\left(\frac34\right)^{2n\frac{p}{q}}\ge \left[3^{\frac{p}{q}}2^{-\frac{p}{q}-1}\left(1+2^{\frac{p(q-1)}{p-q}}\right)^{-(\frac{p}{q}-1)}\right]^{2n},
$$
which implies that
$$
(|\A|\cdot|\B|)^{\frac{p}{q}-1}\le\left[2^{1-\frac{p}{q}}\left(1+2^{\frac{p(q-1)}{p-q}}\right)^{\frac{p}{q}-1}\right]^{2n}.
$$
Thus
$$
|\A|\cdot|\B|\le \left(\frac{1+2^{\frac{p(q-1)}{p-q}}}{2}\right)^{2n}=\left(\frac{1+2^{\frac{\gamma(1+\varepsilon)}{1-\gamma}}}{2}\right)^{2n}.
$$
First let $\varepsilon\downarrow0$ and then $\gamma\downarrow1/2$. Then the right-hand side converges to $(\frac{9}{4})^n$.
\end{proof}

\section{Concluding remarks}\label{sec:remarks}

We compare the entropy proof of Fang and Huang with our hypercontractive argument.

The two proofs start from the same positive kernel. After Sinkhorn scaling, set
$$
\mathbb{P}[X=A,Y=B]=r_As_B3^{-|A\cup B|}.
$$
The row and column normalizations make $X$ and $Y$ uniform on $\mathcal A$ and $\mathcal B$. In~\cite{FangHuang2026}, one sets $U=X\cup Y$ and forms $M\supseteq U$ by adding each element outside $U$ with probability $1/3$. The identity
$$
\mathbb P[M=C,X=A,Y=B]=3^{-n}2^{n-|C|}r_As_B\mathbbm{1}_{\{A\subseteq C\}}
\mathbbm{1}_{\{B\subseteq C\}}
$$
shows that $X$ and $Y$ are conditionally independent given $M$. In the present proof, the same factorization appears in the formula for $\T h_A(C)$. With reference input law $\mu_{1/4}$ for $U$, the operator $\T^{1/4\to1/2}$ averages with respect to the conditional law of $U$ given $M$, that is, the Bayes reverse of the channel $U\mapsto M$.

Theorem~\ref{thm:hyper} also recovers the one-bit entropy inequality used in~\cite{FangHuang2026}. Write $D(P\|Q)$ for relative entropy and $h(x)=-x\log_2x-(1-x)\log_2(1-x)$. Let $P_U$ be a Bernoulli law with $P_U(0)=x$, and let $P_M$ be its image under the channel
$$
\mathbb P[M=1\mid U=1]=1,
\qquad
\mathbb P[M=1\mid U=0]=\frac13.
$$
Then $P_M(0)=\frac{2x}{3}$. Differentiating Theorem~\ref{thm:hyper} at $p=q=1$ and letting $\gamma\downarrow\frac{1}{2}$ gives the relative-entropy contraction
\begin{equation}
 D(P_M\|\mu_{1/2})
\leq\frac12D(P_U\|\mu_{1/4}).
\end{equation}
With logarithms to base $2$, this is equivalent to
\begin{equation}
h(x)+x\log_2 3\leq2h\left(\frac{2x}{3}\right).
\end{equation}
This is the auxiliary inequality used in the proof of~\cite[Lemma 2.2]{FangHuang2026}. The passage from near-$L^1$ hypercontractivity to relative-entropy contraction is part of the standard connection between hypercontractivity and strong data processing; see~\cite{AhlswedeGacs1976,AnantharamEtAl2013,BeigiGohari2018}.

\paragraph{Acknowledgments.}
During the exploratory stage, we used GPT-5.6 Pro to help understand the paper of Fang and Huang. GPT-5.6 Pro suggested an initial version of Theorem~\ref{thm:hyper} in the language of reverse channels. We then observed that this formulation is directly related to the one-sided noise operator defined by Lifshitz, which led to the present formulation of Theorem~\ref{thm:hyper}. We also used GPT-5.6 to assist with polishing the exposition. All other mathematical content in this paper is due to the author.

\end{document}